\documentclass[12 pt]{amsart}

\usepackage{amsmath,amsthm,amssymb,shuffle,stmaryrd,mathtools,enumitem}

\usepackage{ytableau}
\newcommand{\blank}{\phantom{2}}
\usepackage{enumitem}
\usepackage[hmargin=1 in, vmargin = 1 in]{geometry}
\usepackage{hyperref}
\usepackage[capitalize]{cleveref}
\usepackage{mathdots}

\usepackage{tikz-cd}
\tikzcdset{every label/.append style = {font = \footnotesize}}

\newcommand{\ct}{\mathsf{ct}}

\definecolor{darkblue}{rgb}{0.0,0,0.7}

\newcommand{\newword}[1]{\textcolor{darkblue}{\textbf{\emph{#1}}}}

\newcommand{\rect}{\mathrm{rect}}
\newcommand{\Z}{\mathbb{Z}}

\newcommand{\OG}{\mathrm{OG}}

\newcommand{\sPlactic}{\mathbf{S}}

\newcommand{\hookword}{\mathsf{hook}}

\newcommand{\npowerser}[1]{ \!
  \left\langle\mkern-3mu\left\langle#1\right\rangle\mkern-3mu\right\rangle \!}

\newcommand{\sPlacticSchur}{\mathcal{P}}
\newcommand{\sFreeSchur}{\widehat{\mathcal{P}}}

\newcommand{\cN}{\mathcal{N}}

\newcommand{\shSSYT}{\mathrm{ShSSYT}}
\newcommand{\shSSYTprime}{\pr{\shSSYT}}

\newcommand{\diag}{\mathrm{diag}}

\newcommand{\mix}{\mathrm{mix}}

\newcommand{\bbb}{\mathsf{b}}
\newcommand{\ccc}{\mathsf{c}}

\renewcommand{\leq}{\leqslant}
\renewcommand{\geq}{\geqslant}
\newcommand{\pr}[1]{\underline{#1}}
\newcommand{\unpr}[1]{{\overline{#1}}}
\newcommand{\qpr}[1]{\mathbf{#1}}

\newtheorem{thm}{Theorem}[section]
\newtheorem{theorem}[thm]{Theorem}

\newtheorem{lemma}[thm]{Lemma}
\newtheorem{corollary}[thm]{Corollary}

\newtheorem{problem}[thm]{Problem}

\DeclareMathOperator{\rectify}{rect}
\DeclareMathOperator{\stand}{std}

\theoremstyle{remark}
\newtheorem{definition}[thm]{Definition}
\newtheorem{remark}[thm]{Remark}
\newtheorem{examplex}[thm]{Example}
\newenvironment{example}
  {\pushQED{\qed}\examplex}
  {\popQED\endexamplex}

\numberwithin{equation}{section}

\crefname{thm}{Theorem}{Theorems}
\crefname{theorem}{Theorem}{Theorems}

\crefname{lem}{Lemma}{Lemmas}
\crefname{lemma}{Lemma}{Lemmas}

\crefname{corollary}{Corollary}{Corollaries}
\crefname{proposition}{Proposition}{Propositions}
\crefname{problem}{Problem}{Problems}
\crefname{conjecture}{Conjecture}{Conjectures}
\crefname{definition}{Definition}{Definitions}
\crefname{remark}{Remark}{Remarks}

\crefname{example}{Example}{Examples}
\crefname{examplex}{Example}{Examples}

\crefname{equation}{Equation}{Equations}
\crefname{figure}{Figure}{Figures}
\crefname{table}{Table}{Tables}

      \AtBeginDocument{%
   \def\MR#1{}
}

\begin{document}

\title{Mixed jeu de taquin and a problem of Soojin Cho}

\author{Santiago Estupi\~n\'an-Salamanca}
\address[SES]{Department of Combinatorics \& Optimization, University of Waterloo, Waterloo ON N2L 1P3, Canada}
\email{sestupinan@uwaterloo.ca}
\author{Oliver Pechenik}
\address[OP]{Department of Combinatorics \& Optimization, University of Waterloo, Waterloo ON N2L 1P3, Canada}
\email{opecheni@uwaterloo.ca}

\date{\today}


\begin{abstract}
Serrano (2010) introduced the \emph{shifted plactic monoid}, governing Haiman's (1989) \emph{mixed insertion} algorithm, as a type B analogue of the classical plactic monoid that connects jeu de taquin of Young tableaux with the Robinson--Schensted--Knuth insertion algorithm. Serrano proposed a corresponding definition of skew shifted plactic Schur functions.  Cho (2013) disproved Serrano's conjecture regarding this definition, by showing that the functions do not live in the desired ring and hence cannot provide an algebraic interpretation of tableau rectification or of the corresponding structure coefficients. Cho asked for a new definition with particular properties. We introduce such a definition and prove that it behaves as desired. 
We also introduce a new jeu de taquin theory that computes mixed insertion. 
\end{abstract}

\maketitle

\section{Introduction}

The \emph{Schur functions} $s_\lambda$ are a basis of the ring of symmetric functions and their structure coefficients $c_{\lambda,\mu}^\nu$, appearing in the expansions
\[
s_\lambda \cdot s_\mu  = \sum_\nu c_{\lambda,\mu}^\nu s_\nu,
\]
are the \emph{Littlewood--Richardson coefficients}. These objects play important roles in the linear representation theory of symmetric groups (where $s_\lambda$ represents a \emph{Specht module}), in the linear representation theory of general linear groups (where $s_\lambda$ represents a \emph{Schur module}), and in the Schubert calculus of Grassmannians (where $s_\lambda$ represents a \emph{Schubert cycle}). The Littlewood--Richardson coefficients may be studied combinatorially through the interlocking machinery of \emph{jeu de taquin} on \emph{Young tableaux}, the \emph{Robinson--Schensted--Knuth} (RSK) insertion algorithm, and the \emph{plactic monoid}. For exposition of these now-standard ideas, see, for example, the textbooks \cite{Fulton:YT,Manivel}.

The ideas discussed so far may all be considered ``type A.'' In other classical types, the corresponding theories are more complicated and the relations less well developed. The analogues of the usual Schur functions are the \emph{Schur $P$-functions} and \emph{Schur $Q$-functions}, which play similar roles for the projective representation theory of symmetric groups as well as for the Schubert calculus of Lagrangian and maximal orthogonal Grassmannians.

In this setting, \emph{Sagan--Worley} \cite{Sagan,Worley} insertion is an analogue of RSK that is closely connected to a \emph{Sagan--Worley} jeu de taquin. This Sagan--Worley theory does not appear to have a corresponding analogue of the plactic monoid. On the other hand, \emph{mixed} insertion \cite{Haiman} is a different analogue of RSK insertion with a corresponding \emph{shifted plactic monoid} \cite{Serrano}. This latter theory has not been known to have a corresponding jeu de taquin. We provide\footnote{The results of this paper were first announced in the extended abstract \cite{EstupinanSalamanca.Pechenik:FPSAC}.} such a jeu de taquin in \cref{sec:mixedJeudeTaquin}. Our new jeu de taquin computes mixed insertions, despite lacking some of the properties of ordinary jeu de taquin, such as confluence. There is precedent for useful theories of jeu de taquin without confluence, such as the $K$-theoretic jeu de taquin theories of \cite{Thomas.Yong:K, Pechenik.Yong} and we hope that our mixed jeu de taquin will lead to similar such developments.

In type A, the combinatorics of Schur functions can realized through \emph{plactic skew Schur functions}, which plactically encode both jeu de taquin rectification and RSK insertion. This realization yields an algebraic perspective on these combinatorial theories, as well as on the Littlewood--Richardson coefficients. Serrano \cite[p.~379]{Serrano} introduced a notion of \emph{plactic skew Schur $P$-functions}, which he conjectured \cite[Conjecture~2.12]{Serrano} (see also, \cite[Conjecture~7.8]{Cho}) would play an analogous role in other classical types. If true, this would have given an algebraic substitute for the missing jeu de taquin associated to mixed insertion and to the shifted plactic monoid. Unfortunately, Serrano's conjecture was not correct. Precisely, Cho \cite[\textsection 7]{Cho} gave explicit examples to show that Serrano's plactic skew Schur $P$-functions do not lie in the desired ring and hence they cannot serve the desired role in substituting for jeu de taquin or in shedding algebraic light on structure coefficients. Having disproved Serrano's conjecture, Cho asks in \cite[Open Problem~7.12(1)]{Cho} for a new definition that would accomplish Serrano's goals. More precisely, she sets the following task:
\begin{problem}[\cite{Cho}]\label{problem:Cho}
    For $\eta  \leq  \theta$, give a new definition of the plactic skew Schur $P$-function $\sPlacticSchur_{\theta/\eta}$
 so that it belongs to the ring generated by plactic Schur $P$-functions and it describes
the multiplicities of the ordinary $P_\lambda$ in the expansion of $P_{\theta/\eta}$ in a nice way.
\end{problem}
We solve this problem. Interestingly, the key tool of our solution is the Sagan--Worley jeu de taquin theory. By combining this jeu de taquin with shifted plactic theory, we obtain plactic skew Schur $P$-functions that lie in the correct ring and that govern the expansion multiplicities of ordinary Schur $P$-functions, as desired.

\medskip
{\bf This paper is structured as follows.} In \cref{sec:background}, we recall background and give necessary new definitions. Our treatment differs notationally from standard approaches in ways that we believe are clarifying. In \cref{sec:mixedJeudeTaquin}, we introduce mixed jeu de taquin and prove various properties of it. In particular, we show in \cref{thm:mixedRectification} that it computes Haiman's mixed insertions. \cref{sec:Cho} gives a new definition of plactic skew Schur $P$-functions $\sPlacticSchur_{\theta/\eta}$ and shows that these give a solution to \cref{problem:Cho}.

\section{Background and new definitions}\label{sec:background}

A \newword{strict partition} is a decreasing sequence $\theta = (\theta_1 > \theta_2 > \dots > \theta_k > 0)$ of positive integers. We write $|\theta| = \sum_i \theta_i$ and $\ell(\theta) = k$. Associated to a strict partition is a \newword{shifted Young diagram} formed by placing $\theta_i$ boxes in each row $i$ in columns $i, \dots, i + \theta_i - 1$. We use the English convention on Young diagrams, so that row $1$ is the top row and column $1$ is the leftmost column. For example, 
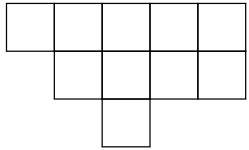
\[
\ydiagram{5, 1 + 4, 2 + 1}
\]
is the shifted Young diagram of the strict partition $(5,4,1)$. We routinely conflate a strict partition with its shifted Young diagram. In this paper, all Young diagrams are shifted.

If $\eta, \theta$ are strict partitions with $\eta_i \leq \theta_i$ for all $i$ (treating $\eta_j = 0$ for $j > \ell(\eta)$), then we have a \newword{skew Young diagram} $\theta / \eta$, the set of boxes of $\theta$ that are not boxes of $\eta$. We identify $\lambda$ with the skew shape $\lambda / \emptyset$. The \newword{diagonal} of a skew Young diagram $\theta / \eta$ consists of the boxes of $\theta / \eta$ in positions $(i,i)$ for some $i$, so for example the diagonal boxes of $(5,4,1)$ are those shaded in the diagram 
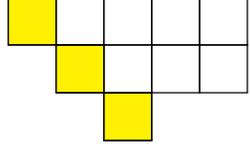
\[
\ydiagram{5, 1 + 4, 2 + 1}  * [*(yellow)]{1, 1+1, 2+1}.
\]

We use the alphabets $\cN$ of positive integers and $\pr{\cN} \cup \unpr{\cN} = \{\pr{1}, \pr{2}, \dots \} \cup \{\unpr{1}, \unpr{2}, \dots \}$ of positive integers decorated with under- and overlines.  We consider this double alphabet with the total ordering 
\begin{equation}\label{eq:alphabet_order}
    \pr{1} < \unpr{1} < \pr{2} < \unpr{2} < \pr{3} < \cdots.
\end{equation}
We call underlined entries \newword{low} and overlined entries \newword{high}. We refer to replacing an entry $\pr{i}$ with $\unpr{\imath}$ as \newword{raising} that entry. We will usually write ${\mathbf i}$ for an unspecified element of $\{\pr{i}, \unpr{\imath} \}$.

    \begin{remark}
    Our notation allows us to easily make a three-way distinction between $\underline{i}$, $\overline{\imath}$, and the underlying integer $i$.     
    
    A more traditional notation is to write $i'$ for $\underline{i}$ and $i$ for $\overline{\imath}$. In this traditional notation, it is awkward to refer unambiguously to the underlying integer, as we frequently need to do.
    \end{remark}

\begin{definition}\label{def:semistandard}
A \newword{semistandard tableau} of shape $\theta/\eta$ is a filling $T$ of the boxes of $\theta/\eta$ with values from $\pr{\cN} \cup \unpr{\cN}$ such that 
\begin{itemize}
    \item rows are weakly increasing from left to right,
    \item columns are weakly increasing from top to bottom,
    \item no low entry appears in any diagonal box,
    \item there is at most one $\unpr{\imath}$ in each column, and
    \item there is at most one $\pr{j}$ in each row.
\end{itemize}
\end{definition}

\begin{example}
    Let $\mu=(1)$ and $\lambda=(7,4,3)$. Then the tableau
    \[
    \begin{ytableau}
        *(lightgray) \empty & \unpr 1 & \unpr 1 & \pr{2} & \pr{4} & \unpr{6} & \pr{7} \\
        \none& \unpr 2 & \unpr 2 & \pr{3} & \pr{4} \\
        \none& \none& \unpr 3 & \pr{4} &\unpr 4
    \end{ytableau}
    \]
    is semistandard and of shape $\lambda/\mu.$
\end{example}

We write $\shSSYT(\theta/\eta)$ for the set of all semistandard tableaux of shape $\theta/\eta$.

The \newword{Schur $P$-function} $P_{\theta / \eta}$ is the symmetric function 
    $$P_{\theta / \eta} \coloneqq \sum_{T\in \shSSYT({\theta / \eta})}{\bf x}^{c(T)}=\sum_{T\in \shSSYT({\theta / \eta})} \prod_{i \in \Z_{>0} }x_i^{c_i(T)} \in \Z \llbracket x_1, x_2, \dots, \rrbracket ,$$
    where $c_i$ counts the total number of entries $\underline{i}$ and $\overline{\imath}$ in $T$.
    We will also need the \newword{Schur $Q$-function} 
    $Q_{\theta / \eta} \coloneqq 2^{\ell(\theta) - \ell(\eta)} P_{\theta / \eta}.$ Note that we can also think of $Q_{\theta / \eta}$ as the generating function for a generalization of shifted semistandard tableaux where we allow low entries in any box; we will consider this perspective further in  \cref{sec:Cho}.

  A \newword{hook subword} of a word $w$ in alphabet $\cN$ is a subword $d\cdot i$ such that $d$ is strictly decreasing and $i$ is weakly increasing. If $w$ is a hook subword of itself, we say $w$ is a \newword{hook word}. For example, $4211688$ is a hook word.
    Let $\lambda$ be a strict partition. Denote by $\hookword(\lambda)$ the set of  words $w=w_1\ldots w_k$, where each $w_i$ is a hook word of length $\lambda_{\ell(\lambda)-i+1}$ and such that, for all $i>1$, $w_i$ is a longest hook subword of $w_{i-1}w_i$. 
    Then the \newword{shifted free Schur function} of shape $\lambda$ is the formal power series in noncommuting variables
\[\sFreeSchur_\lambda \coloneqq \sum_{ w\: : \:  w\in \hookword(\lambda)   }   w \in \Z\npowerser{x_1, x_2, \dots}.\]
(Shifted free Schur functions were first introduced by Serrano \cite{Serrano} using a definition based on \emph{mixed reading words} for certain tableaux; our definition, first stated in \cite{EstupinanSalamanca.Pechenik}, is different from his, but is equivalent.) Note that shifted free Schur functions are not being defined in the skew setting; this fact is the crux of Cho's open problem.

The \newword{shifted plactic monoid}  \cite{Serrano} is the quotient  $\sPlactic$ of the free monoid on $\cN$ by the eight families of relations
    \begin{align*}
        &abdc\sim adbc \quad \text{for all }  a\leq b \leq c<d;  \\
    &acdb\sim acbd \quad \text{for all }  a\leq b < c \leq d;  \\ 
        &dacb\sim adcb \quad \text{for all }  a\leq b < c <d; \\ 
        &badc\sim bdac \quad \text{for all }  a< b \leq c<d;  \\ 
        &cbda\sim cdba \quad \text{for all }  a< b <c \leq d;  \\ 
        &dbca\sim bdca \quad \text{for all }  a< b \leq c<d;  \\ 
        &bcda\sim bcad \quad \text{for all }  a< b \leq c\leq d; \\ 
        &cadb\sim cdab \quad \text{for all }  a\leq  b < c\leq d. 
    \end{align*}
The \newword{shifted plactic Schur function} $\sPlacticSchur_\lambda$ is the image of the shifted free Schur function $\sFreeSchur_\lambda$ under the projection map identifying monomials whose sequences of variable subscripts are equivalent in the shifted plactic monoid. Then $\sPlacticSchur_\lambda \in \Z \llbracket \sPlactic \rrbracket$, the quotient of $\Z\npowerser{x_1, x_2, \dots}$ by the shifted plactic relations on variables. By further projecting to the ring in commuting variables, we recover the ordinary Schur $P$-functions.

Let $\shSSYTprime(\nu/\lambda)$ be the set of shifted tableaux of shape $\lambda$, with diagonal entries possibly low, such that raising the diagonal entries gives an element of $\shSSYT(\nu/\lambda)$. We also call tableaux in $\shSSYTprime(\nu/\lambda)$ the \newword{$Q$-tableaux} of shape $\nu/\lambda$. 

\begin{example}\label{ex:Qtableau}
    Let $\mu=(2)$ and $\lambda=(7,4,2)$. Then the tableau
    \[ U = 
    \begin{ytableau}
       *(lightgray) \empty & *(lightgray) \empty & \unpr 1 & \pr{2} & \pr{5} & \unpr{6} & \pr{7} \\
        \none& \pr{2} & \unpr 2 & \pr{3} & \pr{5} \\
        \none& \none& \unpr 3 & \pr{4}
    \end{ytableau}
    \]
    is a $Q$-tableau of shape $\lambda/\mu$ that is not semistandard because of the $\pr 2$ in the first box of the second row.
\end{example}

Given a $Q$-tableau $T\in \shSSYTprime(\nu/\lambda)$, there is a natural way to represent it as a standard tableau. To wit, list the entries of $T$ so that they form a weakly increasing list
\begin{equation}\label{eq:tableauAlphabet}
    \qpr{a_1}\leq \qpr{a_2} \leq \cdots \leq \qpr{a_n}
\end{equation}
with $\pr{a_i}$ listed before $\pr{a_j}$ if $\pr{a_i}=\pr{a_j}$ and $\pr{a_i}$ is north of $\pr{a_j}$; and $\unpr{a_i}$ listed prior to $\unpr{a_j}$ when $\unpr{a_i}=\unpr{a_j}$ and $\unpr{a_i}$ is west of $\unpr{a_j}$. We can regard this list as an ordered alphabet. Then, there is a unique order preserving bijection $\phi$ between the ordered list of \eqref{eq:tableauAlphabet} and the set $[\unpr n] \coloneqq \{\unpr 1, \dots, \unpr n\}$. We define the \newword{standardization} $\stand(T)$ of $T$ as the tableau resulting from applying $\phi$ to the entries. We say that $T$ is a \newword{standard tableau} if $T = \stand(T)$. 

\begin{example}
    The standardization of the semistandard tableau of \cref{ex:Qtableau} is 
       \[ \stand(U) = 
    \begin{ytableau}
       *(lightgray) \empty & *(lightgray) \empty & \unpr 1 & \unpr 2 & \unpr 8 & \unpr{10} & \unpr{11} \\
        \none& \unpr 3 & \unpr 4 & \unpr 5 & \unpr 9 \\
        \none& \none& \unpr 6 & \unpr 7
    \end{ytableau} \qedhere
    \]
\end{example}

\newword{Mixed insertion} \cite{Haiman} is a process that inserts a letter $\unpr{x} \in \unpr{\cN}$ into a semistandard tableau $T$ as follows. Compare $\unpr{x}$ with the entries of the first row of $T$. If $\unpr{x}$ is weakly greatest, place $\unpr{x}$ at the end of the row and end the process. Otherwise, as in Robinson--Schensted insertion, let $\qpr{y}$ be the least entry of the row that is strictly greater than $\unpr{x}$. Then $\unpr{x}$ replaces $\qpr{y}$ in this row. We next insert $\qpr{y}$. If $\qpr{y} = \unpr{y}$ and $\qpr{y}$ was not bumped from the main diagonal, insert $\unpr{y}$ into the next row in the same fashion. Otherwise either $\qpr{y} = \pr{y}$ or $\qpr{y}$ was bumped from a diagonal box; in these cases, insert $\pr{y}$ into the next column.

\begin{example}
    Let 
    \[
T = 
\begin{ytableau}
  \unpr{1} & \unpr{2} & \unpr{5}\\
  \none &  \unpr{4} & \pr{6}  
\end{ytableau}.
    \]
    We mixed insert $\unpr{1}$ into $T$. First, we obtain 
    \[
\begin{ytableau}
  \unpr{1} & \unpr{1} & \unpr{5}\\
  \none &  \unpr{4} & \pr{6}  
\end{ytableau}
    \]
    and proceed to insert $\unpr{2}$ into the second row.
    This yields
        \[
\begin{ytableau}
  \unpr{1} & \unpr{1} & \unpr{5}\\
  \none &  \unpr{2} & \pr{6}  
\end{ytableau}.
    \]
    Since $\unpr{4}$ was bumped from a diagonal box, we now insert $\pr{4}$ into the third column, obtaining
        \[
\begin{ytableau}
  \unpr{1} & \unpr{1} & \pr{4}\\
  \none &  \unpr{2} & \pr{6}  
\end{ytableau}.
    \]
    Next we insert $\unpr{5}$ into the second row, giving
        \[
\begin{ytableau}
  \unpr{1} & \unpr{1} & \pr{4}\\
  \none &  \unpr{2} & \unpr{5}  
\end{ytableau}
    \]
    and bumping out the $\pr{6}$.
    Finally, the $\pr{6}$ inserts trivially into the next column, which is empty, yielding 
            \[
\begin{ytableau}
  \unpr{1} & \unpr{1} & \pr{4} & \pr{6}\\
  \none &  \unpr{2} & \unpr{5}  
\end{ytableau},
    \]
    the mixed insertion of $\unpr{1}$ into $T$.
\end{example}

The key feature of the shifted plactic monoid is that two words are shifted plactic equivalent if and only if they have the same mixed insertion tableaux. Hence, we may identify shifted plactic classes with semistandard tableaux. We write $[T]_\sPlactic$ for the shifted plactic class of the tableau $T$.

Introduce an extra letter $\bullet$ that is neither over- nor underlined.
A \newword{tableau with holes} is a filling of skew shape in such an extended alphabet such that the entries other than $\bullet$ satisfy the semistandardness conditions of \cref{def:semistandard}. For clarity, we note that as $\bullet$ is neither high nor low, it may appear multiple times in both rows and columns and may appear in diagonal boxes. 
We sometimes refer to boxes outside of the skew shape $\theta/\eta$ as \newword{gaps}.

\section{Mixed jeu de taquin}\label{sec:mixedJeudeTaquin}
Ordinary jeu de taquin is a procedure to turn (non-shifted) skew tableaux into (non-shifted) straight-shape tableaux, in such a way that Robinson--Schensted insertion is respected. To be specific, for any word $w=w_1w_2\cdots w_n$, we have that
\begin{equation}\label{eq:rect}\mathrm{rect}\left(\begin{ytableau}
    \cdot &  \cdot  & \cdots  & w_n   \\
    \cdot &  \cdot  & \iddots & \none \\
    \cdot &  w_2    & \none   & \none \\
    w_1   &  \none  & \none   & \none
\end{ytableau}\right) = P(w) \end{equation}
where $\rect(\bullet)$ denotes the result of jeu de taquin after any choice of slides, and $P(\bullet)$ stands for Robinson--Schensted insertion. 

Unlike the type $A$ case, type $B$ insertion algorithms come in two flavours: Sagan--Worley's insertion and Haiman's mixed insertion. Associated to the first insertion there is a jeu de taquin satisfying a relation analogous to \eqref{eq:rect} and with other desirable properties ordered towards understanding of the cohomology ring $H^\star(\OG(n,2n+1))$. However, a rectification algorithm in the same vein, capable of emulating mixed insertion in place of Sagan--Worley's insertion, has not been developed; nor its potential connection to the structure constants of the Lagrangian Grassmannian explored. In this section, we study one such algorithm that in particular satisfies \eqref{eq:rect} {\it mutatis mutandis}.

\begin{definition}
    Let $T$ be a shifted semistandard tableau of shape $\nu / \lambda$ (recalling that such tableaux have no low entries in diagonal cells). Place a bullet $\bullet$ in each box of the bottom row of $\lambda$. We now swap the bullets past the entries of $T$, deleting them as soon as they have no entry of the tableau to their southeast.

    We say $\qpr{y} \in \bbb$ is \newword{available} if it has a $\bullet$ in $\bbb^\uparrow$ or $\bbb^\leftarrow$.
    We order available entries using the order \eqref{eq:alphabet_order}, breaking ties so that the northmost $\pr{y}$ is least and the westmost $\unpr{y}$ is least.
    If an available entry exists, let $\qpr{y} \in \bbb$ be the least available entry. 

    Below we list $6$ collections of potential \newword{mixed slides}. Apply the unique possible mixed slide from the first collection where a slide is possible:

\begin{equation}\label{eq:diagonal_slides}
\begin{ytableau}
    \bullet & \bullet \\
    \bullet & \qpr{y}
\end{ytableau}
\quad \rightarrow \quad  
\begin{ytableau}
    \qpr{y} & \bullet \\
    \bullet & \bullet
\end{ytableau} \text{,$\hspace{5mm}$or}\hspace{7mm} 
\begin{ytableau}
    \bullet & \bullet \\
    \none   &  \qpr{y}
\end{ytableau} \rightarrow \begin{ytableau}
    \qpr{y} & \bullet \\
    \none   &  \bullet
\end{ytableau};
\end{equation}

\begin{equation}\label{eq:singular_slides}
\begin{ytableau}
    \qpr{x} & \bullet \\
    \bullet & \pr{y}
\end{ytableau} \quad \rightarrow \quad
    \begin{ytableau}
        \qpr{x} & \pr{y} \\
        \bullet & \bullet
    \end{ytableau} \; (\qpr{x} \neq \pr{y}), \quad  \text{or} \quad \;\;
    \begin{ytableau}
    \qpr{x} & \bullet \\
    \bullet & \unpr{y}
\end{ytableau} \quad \rightarrow \quad
    \begin{ytableau}
        \qpr{x} & \bullet \\
        \unpr{y} & \bullet
    \end{ytableau} \; (\qpr{x} \neq \unpr{y});
\end{equation}

\begin{equation}\label{eq:not_singular_slides}
\begin{ytableau}
    \qpr{w} & \qpr{x} & \bullet\\
    \none & \bullet  & \pr{y}
\end{ytableau} \;\;\; \rightarrow \;\;\;
    \begin{ytableau}
        \qpr{w} & \qpr{x} & \bullet\\
        \none & \unpr{y}  & \bullet
    \end{ytableau}, \quad
\begin{ytableau}
    \qpr{x} & \bullet \\
    \bullet & \pr{y}
\end{ytableau} \;\;\; \rightarrow \;\;\;
    \begin{ytableau}
        \qpr{x} & \bullet \\
        \pr{y} & \bullet
    \end{ytableau}, \;\; \text{or} \quad
    \begin{ytableau}
    \qpr{x} & \bullet \\
    \bullet & \unpr{y}
\end{ytableau} \;\;\; \rightarrow \;\;\;
    \begin{ytableau}
        \qpr{x} & \unpr{y} \\
         \bullet & \bullet
    \end{ytableau}; 
\end{equation}

\begin{equation}\label{eq:primedSW}
    \begin{ytableau}
        \bullet & \pr{y} \\
        \none & \unpr{y}
    \end{ytableau}\quad \rightarrow \quad
    \begin{ytableau}
    \unpr{y} & \unpr{y} \\
    \none & \bullet
\end{ytableau}; 
\end{equation}

\begin{align}\label{eq:SW}
    &\begin{ytableau}
        \bullet & \qpr{y} \\
        \none & \blank
    \end{ytableau}  \quad \rightarrow \quad
     \begin{ytableau}
        \unpr{y} & \bullet \\
        \none & \blank
    \end{ytableau}, \quad
    \begin{ytableau}
    \qpr{x} & \bullet \\
    \none & \unpr{y}
\end{ytableau} \quad \rightarrow \quad
\begin{ytableau}
    \qpr{x} & \pr{y} \\
    \none & \bullet
\end{ytableau}; \\
 &\begin{ytableau}
        \bullet & \qpr{y} 
    \end{ytableau}  \quad \rightarrow \quad
     \begin{ytableau}
        \qpr{y} & \bullet 
    \end{ytableau}, \quad
 \begin{ytableau}
        \bullet \\ \qpr{y} 
    \end{ytableau}  \quad \rightarrow \quad
     \begin{ytableau}
        \qpr{y} \\ \bullet 
    \end{ytableau}.
\end{align}

Repeat this operation on the new least available entry (possibly the same as the last least available entry), unless no entries are available. The result is a tableau of shape $\nu'/ \lambda'$. Place a bullet in each box of the bottom row of $\lambda'$ and repeat until obtaining a tableau of straight shape. The result is the \newword{mixed rectification} $\mathrm{rect}_\mix(T)$ of $T$.

The moves of \eqref{eq:diagonal_slides} are called \newword{diagonal slides}. Essentially, non-diagonal slides will simulate row- or column-displacements during mixed insertion of a word, whereas entries that only partake of diagonal slides will correspond to entries unaffected by such displacements. 
\end{definition}

Note that some of these slides match the local rules of the Sagan--Worley jeu de taquin \cite{Sagan, Worley}, but in general they behave quite differently.
\begin{example}
    Let 
    \[
    T = \begin{ytableau}
    \blank & \blank & \blank & \blank & \unpr{1}\\
    \none & \blank & \unpr{2} & \unpr{4} \\
    \none & \none & \unpr{3} 
    \end{ytableau}. 
    \]
    Then we compute the mixed rectification by the following steps:
    \[\def\arraystretch{2.6}
    \begin{array}{ccccccc}
        \begin{ytableau}
    \blank & \blank & \blank & \blank & \unpr{1} \\
    \none & \bullet & \unpr{2} & \unpr{4} \\
    \none & \none & \unpr{3} 
    \end{ytableau} & \rightarrow &  \begin{ytableau}
    \blank & \blank & \blank & \blank & \unpr{1} \\
    \none & \unpr{2} & \bullet & \unpr{4} \\
    \none & \none & \unpr{3} 
    \end{ytableau}  &\rightarrow &
    \begin{ytableau}
    \blank & \blank & \blank & \blank & \unpr{1}\\
    \none  & \unpr{2} & \pr{3} & \unpr{4}\\
    \none & \none & \bullet 
    \end{ytableau} & \rightarrow &
    \begin{ytableau}
    \bullet & \bullet & \bullet & \bullet & \unpr 1 \\
    \none &  \unpr{2} & \pr{3} & \unpr{4} 
    \end{ytableau}  \\[20pt]
    \begin{ytableau}
    \unpr{1} & \bullet & \bullet & \bullet \\
    \none &  \unpr{2} & \pr{3} & \unpr{4} 
    \end{ytableau} & \rightarrow & 
    \begin{ytableau}
    \unpr{1} & \pr{2} & \bullet & \bullet \\
    \none &  \bullet & \pr{3} & \unpr{4}  
    \end{ytableau} & \rightarrow &
    \begin{ytableau}
    \unpr{1} & \pr{2} & \pr{3} & \bullet  \\
    \none &  \bullet  & \bullet & \unpr{4}  
    \end{ytableau} & \rightarrow 
    & \begin{ytableau}
    \unpr{1} & \pr{2} & \pr{3} \\
    \none & \unpr{4}
    \end{ytableau}.
    \end{array}
    \]
\end{example}

The following lemma gives a useful structural property of mixed rectification. 

\begin{lemma}\label{prop:noPrimesDiagonals}
    Let $T \in \shSSYT(\nu/\mu)$. Then, throughout the mixed rectification of $T$, no low entry occupies a diagonal cell.
\end{lemma}
\begin{proof}
    All mixed slides that move an entry into a diagonal position ensure that it is high. Accordingly, if $T$ has no low entries in diagonal positions to start with, its mixed rectification is also free of low entries in diagonal positions. 
\end{proof}

\begin{definition}
    Let $T$ be a tableau with holes and $\qpr{y}$ an entry of $T$ that is least available. Then $\qpr{y}$ \newword{slides completely} from $\bbb$ to $\mathsf{c}$, if there is a sequence of slides involving $\qpr{y}$ as the least available entry that starts at $\bbb$, ends at $\mathsf{c}$, and so that subsequent slides in the mixed rectification process do not change the position of $\qpr{y}.$
    We also say that $\qpr{y}$ slides completely if there exist $\bbb$ and $\mathsf{c}$ such that it slides completely from $\bbb$ to $\mathsf{c}.$
\end{definition}

\begin{lemma}\label{lem:southeast}
Consider mixed rectification on a tableau $T \in  \shSSYT(\nu/\mu)$.
    If an entry $\qpr{x}$ has not yet moved during mixed rectification, then all entries southeast of $\qpr{x}$ are at least as large as $\qpr{x}$.
\end{lemma}
\begin{proof}
    Before mixed rectification, the lemma holds by semistandardness of $T$. Since no slide moves an entry South or East, all slides of mixed rectification that do not change the value of the entries involved preserve this property.
    
    Of the three remaining slides, two (the first slides of \eqref{eq:not_singular_slides} and \eqref{eq:SW}) increase the value of the entry being slid and so also preserve the property.
    
    It remains to consider the second move of \eqref{eq:SW}, by which $\unpr{y}$ moves to the cell $\bbb$ containing a bullet and turns into $\pr{y}$. Since by assumption, $\qpr{x}$ has not moved, we may assume that before any entry was slid, $\bbb \ni \qpr{z}$ with $\qpr{z} \leq \pr{y}$. At that moment, $\qpr{z}$ is southeast of $\qpr{x}$ and $\qpr{x}\leq \qpr{z}$, so that $\qpr{x}\leq \pr{y}.$ But by \cref{prop:noPrimesDiagonals}, $\qpr x = \unpr{x}$, so $\qpr{x} < \pr{y}.$
\end{proof}

\begin{lemma}\label{lem:slidingCompletely}
    Suppose that an entry $\qpr{x}$ becomes the least available entry at some point during the mixed rectification of a semistandard tableau $T \in \shSSYT(\nu/\mu)$. Then, once $\qpr{x}$ slides, it slides completely. 
\end{lemma}
\begin{proof}
    We proceed by induction on the number of complete slides that have been performed. 
    
    For the base case, note that every slide involves only one entry $\qpr{x}$, which corresponds to the least available by definition, so that $\qpr{x}$ will continue to slide until no longer available or at some point it changes value. Considering the list of possible moves, we can then restrict our attention to slides in which $\qpr{x}$ is slid into a diagonal position. 

    Since the first complete slide takes place just after gaps were created in the last row of $\mu$, there is only one diagonal position $\bbb$ occupied by a gap. That is, if $\qpr{x}$ changes from being $\pr{x}$ to $\unpr{x}$ as it is slid to the cell $\bbb$, then it will not move again because $\bbb$ is in the first row containing gaps. This completes the base case.

    Suppose now that the lemma holds after $n$ complete slides and a new slide is about to be performed. Again, let $\qpr{x}$ be the least available entry and consider the first slides of \eqref{eq:not_singular_slides} and of \eqref{eq:SW}, as they are the only ones potentially increasing the value of $\qpr{x}.$

    Throughout the $(n+1)$-st complete slide, we have that entries northwest of $\qpr{x}$ have either not moved at all or slid completely by the inductive hypothesis. It follows then that no entry northwest of $\qpr{x}$ is available, either because it has slid completely or else because it would have been least available instead of $\qpr{x}$ as per \cref{lem:southeast}. But then, no entry northwest of $\qpr{x}$ will move in the future. Therefore, if $\qpr{x}$ participates in the first slide of \eqref{eq:not_singular_slides}, it arrives to a diagonal position from which it will not move again. 

    On the other hand, performing the first slide of \eqref{eq:SW} implies that the slides of \eqref{eq:diagonal_slides} were infeasible. In particular, if $\qpr{x}\in\ccc$, then either $\ccc^\uparrow$ or $\ccc^{\uparrow \leftarrow}$ is occupied immediately before $\qpr{x}$ is slid into the diagonal position. If $\ccc^{\uparrow \leftarrow}$ is occupied, recall that its entry will not move again, being northwest of $\qpr{x}$; hence once $\qpr{x}$ enters the diagonal cell, it slides completely.

    If instead $\ccc^\uparrow$ is occupied but $\ccc^{\uparrow \leftarrow}$ is not, then $\ccc^\uparrow$ is available. But again, \cref{lem:southeast} forces $\ccc^\uparrow\leq \ccc$, contradicting that $\ccc$ is least available and not $\ccc^\uparrow$.    
\end{proof}

\begin{lemma}\label{lem:leastAvailableMonotone}
     Let $\mu = (\mu_1 \geq \dots \geq \mu_k > 0)$ and let $T \in \shSSYT_\bullet(\nu/\mu)$ be a semistandard tableau with holes whose bullets occupy the skew shape $\theta / \mu$ for some $\theta = (\mu_1 \geq \dots \geq \mu_k \geq \theta_{k+1} > 0)$. Suppose that $\qpr{x}<\qpr{y}$ are entries that become available during the rectification of $T.$ Then $\qpr{x}$ becomes least available before $\qpr{y}$ becomes least available. 
\end{lemma}
\begin{proof}
    We restrict our attention to the rows of $T$ strictly below the last row of $\mu$. The remaining entries never become available so we can safely exclude them from the analysis. 
    
    Suppose by contradiction that $\qpr{x}<\qpr{y}$ but that $\qpr{y}$ becomes least available before $\qpr{x}$. Without loss of generality, take $\qpr{y}$ the first such entry in the mixed rectification process to falsify the lemma, and take $\qpr{x}$ the first such entry to become least available with respect to $\qpr{y}$. The instant $\qpr{y}$ becomes least available, i.e., the configuration after a complete slide that has $\qpr{y}$ least available for the first time, must not also have $\qpr{x}$ be available. Moreover, as $\qpr{x}\in \bbb$ becomes least available later, we have now that at least one of $\bbb^\uparrow$ or $\bbb^\leftarrow$ is occupied by an entry that has not moved, but will move later. 
    
    Next, note that there are two types of entries weakly northwest of $\qpr{x}$: those that have not moved and those that have. (Empty entries never move nor affect the displacement of bullets or which entries become available, so we may ignore them.) For those that have moved, we know by \cref{lem:slidingCompletely} that they will not participate in a slide again. On the other hand,  those that have not moved fall again in two categories: available and not available. Entries $\qpr{z}$ that have not moved and are available, satisfy $\qpr{z}\leq \qpr{x}$ by \cref{lem:southeast}. Hence, recalling that $\qpr{y}$ is least available at this moment, we obtain $\qpr{y}\leq \qpr{z}\leq \qpr{x}$, a contradiction. Accordingly, there are no entries that have not moved and are available at this time. Thus all entries weakly northwest of $\qpr{x}$ are not available, but this is only possible if there are no bullets weakly northwest of $\qpr{x}$. Then $\qpr{x}$ will not become least available at any later time, contradicting that it becomes least available after $\qpr{y}.$
\end{proof}

The following is a restatement of \cref{lem:leastAvailableMonotone} that is convenient for later use.

\begin{corollary}\label{cor:leastAvailableMonotone}
    Let $T$ be a semistandard tableau with holes. If $\qpr{w}$ becomes least available before $\qpr{z}$ does during the rectification of $T$, then $\qpr{w}\leq \qpr{z}.$
\end{corollary}
\begin{proof}
    It is immediate from \cref{lem:leastAvailableMonotone}.
\end{proof}

\begin{lemma}\label{lem:well-defined}
Mixed rectification of $T\in \shSSYT(\nu/\lambda)$ is a well-defined process that produces a semistandard tableau of straight shape. Furthermore, the intermediate steps are semistandard tableaux with holes.
\end{lemma}
\begin{proof}
By \cref{prop:noPrimesDiagonals}, we know that all the intermediate steps have no high entries on the diagonal. Hence it remains to check the increasingness conditions.

We will show that each intermediate tableau is semistandard by induction in the number of mixed rectification steps performed. For one step the result is clear by inspection of the mixed slides. Assume that the lemma holds for $m\leq n$ steps.

Calculating the mixed rectification of $T$ involves creating the bottommost possible row of gaps, moving available entries until all the gaps disappear, and repeating the process. Moreover, every such complete iteration results in a new tableau with holes that is in particular a skew tableau.

Suppose that $n$ slides have been performed, creating the semistandard tableau with holes $U$. Consider the effect of a  slide on $U$. 

A diagonal slide of the form
\[\begin{ytableau}
    \bullet & \bullet \\
    \bullet & {\mathbf{y}}
\end{ytableau} \quad \rightarrow \quad \begin{ytableau}
    {\mathbf{y}} & \bullet \\
    \bullet & \bullet
\end{ytableau}\]
moves $\qpr{y}$ from a box $\bbb$ to a box $\bbb^{\uparrow,\leftarrow}$. If $\qpr{y}$ becomes the westmost entry of its new row, we only need to worry about letters $\qpr{z}$ east of $\bbb^\uparrow$; otherwise, it suffices to note that entries left of $\bbb^{\uparrow,\leftarrow}$ became least available before $\qpr{y}$ and \cref{cor:leastAvailableMonotone} guarantees that the order is preserved. 
 
If there is an entry $\qpr{z}$ to the right of $\bbb^\uparrow$, then $\qpr{z} \geq \qpr{y}$ because otherwise $\qpr{y}$ would not be the least available entry. If there is one gap or more to the east of $\bbb^\uparrow$, then the first entry $\qpr{a}$ that is not a gap will either be slid before $\qpr{y}$ so that it is no longer in the same row, or end next to $\bbb^\uparrow$. In the former case the entry does not matter for our considerations, while in the latter we can repeat the argument with $\qpr{z}.$ In either case, the resulting tableau with holes satisfies weak row increasingness. Moreover, if $\qpr{z} = \qpr{y}$, then the fact that $\qpr{y}$ is the least available entry implies that $\qpr{y} = \unpr{y}$, so we do not have two equal low entries in the same row.

Similarly, if there is an entry below $\bbb^\leftarrow$, though not necessarily contiguous, then either it is equal to $\qpr{y}=\pr{y}$ and monotonicity is respected; or it must be greater than $\qpr{y}$ if the latter is to be available. Hence, the resulting tableau is semistandard. 

Analyzing other diagonal slides, we see that just from the fact that $\qpr{y}$ is currently available, it can also be deduced that the final configuration is monotone by rows and columns. 

For slides of the form 
\[\begin{ytableau}
    \qpr{x} & \bullet \\
    \bullet & \pr{y}
\end{ytableau} \quad \rightarrow \quad
    \begin{ytableau}
        \qpr{x} & \pr{y} \\
        \bullet & \bullet
    \end{ytableau},\]
there are two cases: either $\qpr{x}$ has not moved and $\qpr{x}\leq \pr{y}$ by \cref{lem:southeast}, or it has moved and again $\qpr{x}\leq \pr{y}$ because of \cref{cor:leastAvailableMonotone}. In both cases, this results in a configuration that is monotone by rows. Monotonicity by columns is clear because only the positions of the gap and $\pr{y}$ have been changed, and then such a slide preserves that the tableau is semistandard with holes. 

The analysis of the remaining slides that do not change the value of the entry is analogous.

For the first slide in \eqref{eq:not_singular_slides}, column monotonicity is clear, while row monotonicity holds by the inductive hypothesis and the fact that no two low entries of the same value can lie in the same row. The first slides of \eqref{eq:SW} and \eqref{eq:primedSW} also result in a semistandard tableau with holes by the inductive hypothesis.

It remains then to consider the last slide of \eqref{eq:SW}. Column monotonicity is follows in this case from the inductive hypothesis, since entries above the $\unpr{y}$ are at most $\pr{y}$. Hence, let us establish row monotonicity. There are two possibilities for $\qpr{x}$: either it has not moved yet or it has completely slid. If the former, then as $\unpr{y}$ has not moved either, both were part of the original tableau $T$. Since both were occupying diagonal cells, it follows that $\qpr{x}=\unpr{x}<\unpr{y}$ and the slide results in a configuration monotone by rows. If the latter, by virtue of \cref{cor:leastAvailableMonotone}, $\qpr{x} \leq \unpr{y}$ and the fact that the slide of \eqref{eq:primedSW} was not performed indicates that $\qpr{x}<\unpr{x}$ so that monotonicity is respected once more. 
\end{proof}

    We denote the skew tableau containing $T$ and $\unpr y$ in the antidiagonal by
    \[ T\oplus \unpr y \coloneqq \begin{ytableau}
            \blank & \cdots & \unpr y\\
            \none &  T & \none 
        \end{ytableau}. \]

Consider the mixed rectification of $T'$ defined by
     $T'\coloneqq  T \oplus \unpr{y}.$ Let $U$ be one of the intermediate configurations obtained. We define the \newword{nucleus} of $U$ to be the largest set of filled boxes in $U$ that contains the top-leftmost box and is a straight-shaped tableau.

\begin{lemma}
    Suppose that an entry $\qpr{x}$ becomes the least available entry at some point during the mixed rectification of 
     $T'\coloneqq  T \oplus \unpr{y}.$
     Then $\qpr{x}$ participates in a sequence of slides until it joins the nucleus and becomes unavailable. After this, it is never again available.
\end{lemma}
\begin{proof}
    Let us establish the lemma by induction on the number of complete slides. The base case is clear as the nucleus after one complete slide is a single diagonal cell. 
    
    Suppose that $n$ complete slides have been performed and assume by contradiction that the next complete slide, displacing $\qpr{x}$, does not result in $\qpr{x}$ joining the nucleus. Then, $\qpr{x}$ is still available when it reaches its final position $\bbb$, but an entry $\qpr{y}\neq \qpr{x}$ becomes least available immediately after. Since $\qpr{x}$ was least available prior to this moment, it must have slid in such a way as to make $\qpr{y}$ available. Thus, $\qpr{y}\geq \qpr{x}$ by \cref{cor:leastAvailableMonotone}, and if $\qpr{y}=\qpr{x}$, we have that the entry $\qpr{x}$ is westmost if low, northmost if high. This contradicts the fact that $\qpr{y}$ will be least available after $\qpr{x}$ moves into $\bbb$, because $\qpr{x}$ is still available.
\end{proof}

    The moral of the following lemma is that the only non-trivial slides are slides that happened as a consequence of a previous vertical or horizontal slide.

\begin{lemma}\label{lem:nonTrivialSlides}
Suppose that an entry $\qpr{x}$ becomes the least available entry at some point during the mixed rectification of  $T'\coloneqq   T \oplus \unpr{y} ,$ where $T$ is a tableau of straight shape.

        Let the first slide of $\qpr{x}$ be from $\bbb$ to $\bbb^\leftarrow$ or $\bbb^\uparrow$. Then, when $\qpr{x} \in \bbb$, we have $\qpr{z} \in \bbb^{\leftarrow\uparrow}$, where $\qpr{z}$ reached $\bbb^{\leftarrow\uparrow}$ by a sequence of purely horizontal or purely vertical slides.  
\end{lemma}
\begin{proof}
    Suppose that $\qpr{x}$ becomes the least available entry while sitting in box $\bbb$. Since the first slide is either horizontal or vertical, the box $\bbb^{\uparrow \leftarrow}$ cannot be empty; otherwise, either $\qpr{x}$ performs a diagonal slide, or \cref{lem:well-defined} implies that one of the entries adjacent to $\qpr{x}$ is least available. Say $\qpr{z} \in \bbb^{\uparrow \leftarrow}$. 
    
    Note that before $\qpr{x}$ becomes least available, it was motionless in $\bbb$ by \cref{lem:slidingCompletely}.
    Then, the final slide bringing $\qpr{z}$ to $\bbb^{\uparrow \leftarrow}$ cannot have been diagonal, since $\qpr{x} \in \bbb$. There are two cases for the last slide of $\qpr{z}$.
    If it was vertical, then by the increasingness conditions of \cref{lem:well-defined}, all previous slides of $\qpr{z}$ must also have been vertical. Indeed, to perform a horizontal or diagonal slide, $\qpr{z}$ would have been in the column of $\qpr{x}$ at some point, so that by \cref{lem:well-defined} it would have held that $\qpr{x}\leq \qpr{z}$ at that moment; but then \cref{lem:slidingCompletely} implies that at that moment the positions $\qpr{z}$ will pass through are bullets, meaning that $\qpr{x}$ is available, hence should have been least available at that moment, not $\qpr{z}$, a contradiction. 

    Suppose instead that the last slide was horizontal. A similar analysis and the increasingness conditions of \cref{lem:well-defined} likewise show that all previous slides of $\qpr{z}$ are forced to be horizontal.
\end{proof}

Mixed rectification can be made to simulate mixed insertion by properly positioning the tableau $T$ and entry $\unpr{y}$ to be inserted, into the configuration $T\oplus \unpr{y}$, and rectifying. 

Entries in the mixed rectification process that perform at least one non-diagonal slide will correlate to bumped entries during mixed insertion. That being said, an entry may sometimes perform diagonal slides in addition to non-diagonal slides, so that a more rigorous concept is needed to capture such entries inside mixed rectification.

\begin{definition}
    Let $\qpr{x}$ be an entry of $T\oplus \unpr{y}$. Then $\qpr{x}$ is singular if at some moment of the mixed rectification of $T\oplus \unpr{y}$
    \begin{itemize}
        \item $\qpr{x}=\unpr{y}$; or,
        \item the first time $\qpr{x}$ moves, it does so from a cell which is in the same diagonal as a singular entry, and immediately southeast of it. 
    \end{itemize}
\end{definition}

    As it will become clear, the first slide of any singular entry must also be singular, so that an alternative characterization may be given to that effect. 

    \begin{theorem}\label{thm:mixedRectification}
        Let $T$ be an arbitrary shifted Young tableau and $\unpr y$ a high letter. Then performing the mixed insertion of $\unpr y$ on $T$ results in the tableau 
        $\mathrm{rect}_\mix \left( T\oplus \unpr y \right)$.
    \end{theorem}
    \begin{proof}

    The mixed insertion of $\unpr{y}$ into $T$ involves a sequence of tableaux $(T_i)_{i=1}^{d}$ such that $T_i$ can be obtained from $T_{i-1}$ by a row or column insertion. For example, $T_1=T$ and $T_2$ is the result of row-inserting $\unpr{y}$ into $T_1$; to be clear, not the complete row-insertion, but its first step. The last element of that sequence, $T_d$, is the $P$-tableau of the mixed insertion. 

    We prove that the sequence $(T_i)_i$ can be read from mixed rectification in the sense that for every intermediate configuration's nucleus $U$, there exists $i$ such that after sliding of a singular entry $U$ is a subtableau of $T_i$; and so that if $i$ is minimal with that property, the change between $T_{i-1}$ and $T_i$ is the singular entry which is now part of the nucleus. Furthermore, noting that the nucleus at the end of mixed rectification contains all of the entries of $T\oplus \unpr{y}$, we observe that if throughout the process the nucleus of $U$ is always a subtableau of $T_i$ then the theorem follows. 

    To establish this fact, we show by induction on the number of singular entries of the rectification that the $i$-th singular entry (the first one is $\unpr{y}$) only slides horizontally (vertically) after the entries in the corresponding row (column) of $T_i$ less than it, have been slid to that row (column); that after it has landed either the insertion process has been completed, i.e., $T_{i+1}=T_d$, or there is an entry diagonal to it (immediately next in the same diagonal); that the entry diagonal to it is both the new singular entry and the entry displaced by the corresponding insertion to $T_i$; and that any other entries which move during the process agree with their positions in $T_i.$

    For the base case, consider a rectification process that has only one singular entry. Since by design we begin with a tableau with holes having a copy of $T$ and a high entry above, and the high entry is singular by definition, it is sufficient to prove that if it does not land in the end of the first row then there is at least one more singular entry. Indeed, if it lands in the end of the first row, then it is greater than all other letters in the first row of $T$ by the mixed rectification process, and so the mixed insertion of $\unpr{y}$ also appends it to the right boundary of the first row leaving everything else unchanged.

    Suppose then by contradiction that at the end of mixed rectification $\unpr{y}$ is not rightmost in the first row. Since all the entries in the copy of $T$ inside $U$ slide diagonally unless $\unpr{y}$ is northwest and immediately diagonal to one of the entries in the first row, we can assume that is the case for one of them, say $\qpr{x}$. But then, as smaller entries are slid before bigger ones per \cref{cor:leastAvailableMonotone}, we have that when it is the turn of $\qpr{x}$ to be slid it is diagonal to an entry that is singular. That is, $\qpr{x}$ is singular. 

    Let $(\qpr{y_i})_i$ be the list of singular entries as they appear in the mixed rectification process, where in particular $\qpr{y_0}=\unpr{y}$. Assume that mixed rectification up to sliding the entry $\qpr{y_{i-1}}$ is in agreement with mixed insertion, meaning that the inductive hypothesis is satisfied. Then the new singular entry $\qpr{y_i}$ was bumped when $\qpr{y_{i-1}}$ was inserted into $T_{i-1}$ during mixed insertion.

    {\sf Case 1 ($\qpr{y_i}$ is a low letter):}
    Let $c$ be the column of $\qpr{y_i}=\unpr{y_i}$ in $U$. Suppose $\qpr{x}$ is an entry that is north of $\unpr{y_i}$ in $T_{i+1}$, i.e., after it has been column inserted during mixed insertion. There is a copy of $\qpr{x}$ in $U$ and it belongs either to the nucleus of $U$, or to the set of entries which have not been slid. 
    
    If the former, then $\qpr{x}$ moved as a singular entry or as a non-singular entry. Non-singular entries are transported diagonally, one unit, to their final position in $U$, whereby they reach their original position in $T$, hence in $T_{i+1}$. Indeed, by the inductive hypothesis, $\qpr{x}$ was not bumped during mixed insertion prior to $\unpr{y_i}$, and as $\qpr{x}$ is north of $\unpr{y_i}$ in $T_{i+1}$, we also have $\qpr{x}< \unpr{y_i}$ and so neither was $\qpr{x}$ bumped by $\unpr{y_i}$'s insertion. Accordingly, $\qpr{x}$ slides into a cell matching its position in $T$ and $T_i$.
    Singular entries are bound by the inductive hypothesis to agree with $T_j$ for some $j\leq i$, hence to agree also with $T_i$, because once an entry has been slid it is not slid again. 
    
    If the latter, then the entries are still in their respective positions as elements of $T$, albeit displaced diagonally one unit south and one unit east. Noting again that they correspond to entries that have not been bumped, they are also in their corresponding positions as elements of $T_i$. That is, in both cases and after $\qpr{y_{i-1}}$ has been slid, $U$ reflects the positions that the entries have in $T_i$. 

    Note that prior to $\unpr{y_i}$ sliding, all entries inside the nucleus of $U$, and in column $c$, are less than $\unpr{y_i}$ because smaller entries are mixed rectified before bigger entries. Moreover, all entries north of $\unpr{y_i}$ in $T_{i+1}$ are north of $\unpr{y_i}$ in $U$ immediately before it is slid. Indeed, we already argued that all the entries in the column $\unpr{y_i}$ is being inserted to in $T_i$ were either in the nucleus of $U$ or in column $c+1$ and in the right order. Then, when $\unpr{y_i}$ completes its traverse northwards it must do so diagonal of the first entry $z$ in column $c$ of $T_i$ which was not less than $\unpr{y_i}$. That is, the entry $\qpr{y_{i+1}}$ which is displaced by the respective column insertion. Or, if there is no such $\qpr{y_{i+1}}$, diagonal to no entry corresponding to the end of the insertion and the mixed rectification. That is, $\unpr{y_i}$ finishes in the same position as the column insertion would have it, and a new singular entry is created if applicable: the letter displaced by the column insertion.
    
    There may be other letters which mix-slide albeit not as singular entries. We claim that they slide diagonally (one unit north, one unit west). First note that at any point of mixed rectification there is at most one singular entry, because we are completely sliding a letter before sliding the next one (\cref{lem:slidingCompletely}), and once a singular entry is part of $U$'s nucleus, there is at most one entry diagonal to it that has not been slid. Crucially, singular entries are the only entries which occupy cells immediately diagonal (one unit north, one unit west) to entries of $U$ that have not moved. Recalling that entries diagonal to singular entries are also singular, this means that all non-singular entries slide diagonally. 
    Thus non-singular entries $\qpr{w}$ with $\qpr{w}\leq \qpr{y_{i+1}}$ slide diagonally, and as they are not entries bumped during the mixed insertion, they rest in cells corresponding to their original position in $T$ and hence also the appropiate position in $T_i$. Then, the resulting nucleus is a subtableau of $T_{i+1}$.

    {\sf Case 2 ($\qpr{y_i}$ is an high letter):}
This case is completely analogous to the case of $\qpr{y_i}$ low, except for the case where it may slide past a diagonal entry. In that case, the diagonal entry becomes the new singular entry.
\end{proof}
    
\begin{example}
        Let us see \cref{thm:mixedRectification} in action. Consider the word $\unpr{7}\unpr{3}\unpr{9}\unpr{4}$. Its mixed insertion equals the mixed rectification process below:  
    \begingroup
    \addtolength{\jot}{1em}
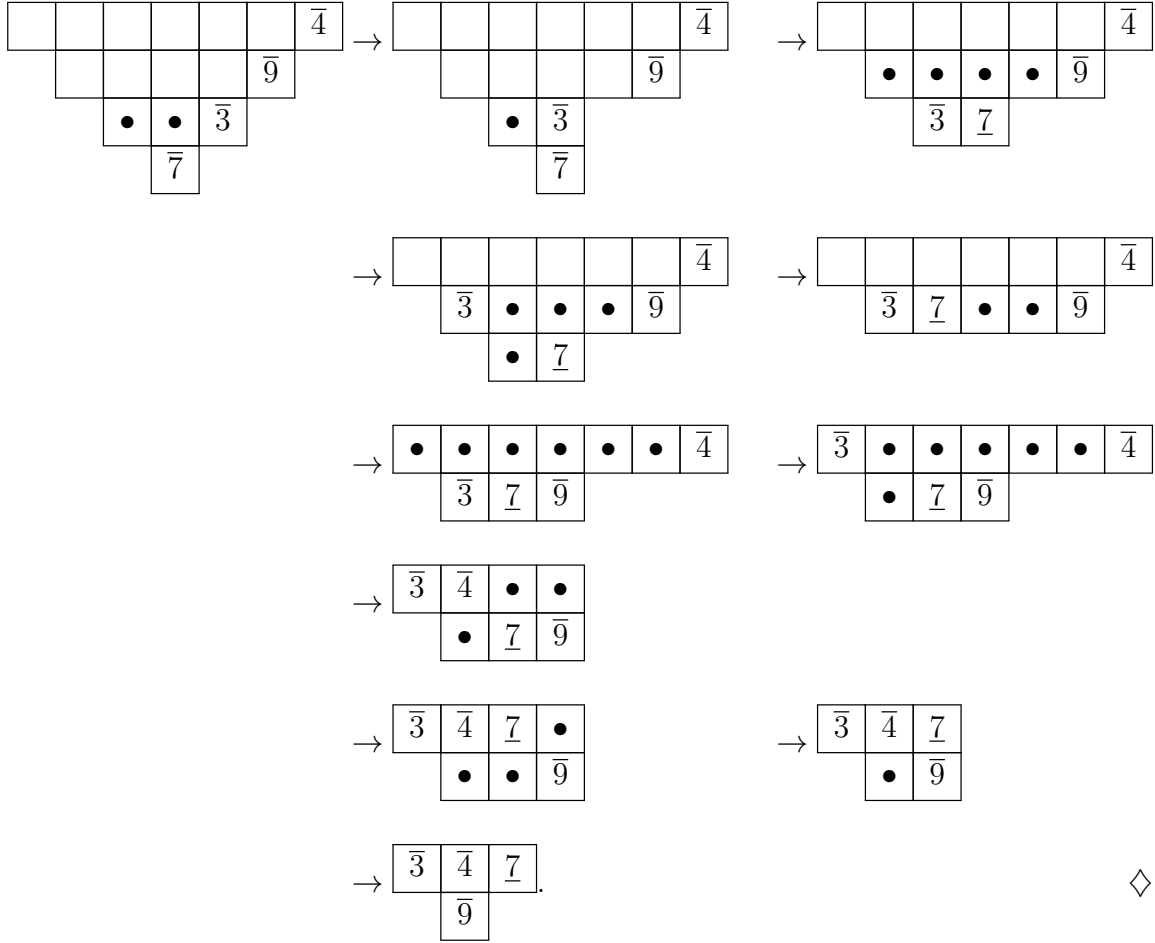
    \begin{align*}
        \begin{ytableau}
            \blank & \blank & \blank & \blank & \blank & \blank & \unpr{4} \\
            \none& \blank & \blank & \blank &\blank & \unpr{9}\\
            \none& \none & \bullet & \bullet & \unpr{3}\\
            \none& \none & \none   & \unpr{7}
        \end{ytableau} &\rightarrow
        \begin{ytableau}
            \blank & \blank & \blank & \blank & \blank & \blank &\unpr{4} \\
            \none& \blank & \blank & \blank &\blank & \unpr{9}\\
            \none& \none & \bullet & \unpr{3} \\
            \none& \none & \none   & \unpr{7}
        \end{ytableau}
        &\rightarrow 
        \begin{ytableau}
            \blank & \blank & \blank & \blank & \blank & \blank &\unpr{4} \\
            \none& \bullet & \bullet & \bullet &\bullet & \unpr{9}\\
            \none& \none & \unpr{3} & \pr{7} \\
        \end{ytableau}\\
        &\rightarrow 
        \begin{ytableau}
            \blank & \blank & \blank & \blank & \blank & \blank &\unpr{4} \\
            \none& \unpr{3} & \bullet & \bullet &\bullet & \unpr{9}\\
            \none& \none & \bullet & \pr{7} \\
        \end{ytableau}
        &\rightarrow 
        \begin{ytableau}
            \blank & \blank & \blank & \blank & \blank & \blank &\unpr{4} \\
            \none& \unpr{3} & \pr{7} & \bullet &\bullet & \unpr{9}\\
        \end{ytableau}\\
        &\rightarrow 
        \begin{ytableau}
            \bullet & \bullet & \bullet & \bullet & \bullet & \bullet &\unpr{4} \\
            \none& \unpr{3} & \pr{7} & \unpr{9}
        \end{ytableau}
        &\rightarrow 
       \begin{ytableau}
            \unpr{3} & \bullet & \bullet & \bullet & \bullet & \bullet &\unpr{4} \\
            \none&  \bullet & \pr{7} & \unpr{9} 
        \end{ytableau}\\
        &\rightarrow 
        \begin{ytableau}
            \unpr{3} & \unpr{4} & \bullet & \bullet   \\
            \none& \bullet & \pr{7} & \unpr{9}  
        \end{ytableau}\\
        &\rightarrow \begin{ytableau}
            \unpr{3} & \unpr{4} & \pr{7} & \bullet  \\
            \none& \bullet & \bullet & \unpr{9}  
        \end{ytableau}
        &\rightarrow 
        \begin{ytableau}
            \unpr{3} & \unpr{4} &  \pr{7} & \none  & \none & \none & \none  \\
            \none& \bullet & \unpr{9}    
        \end{ytableau}\\
        &\rightarrow \begin{ytableau}
            \unpr{3} & \unpr{4} & \pr{7}    \\
            \none& \unpr{9}  
        \end{ytableau}. &\phantom{\rightarrow} \qedhere 
    \end{align*}
    \endgroup 
\end{example}

\section{A solution to Cho's problem}\label{sec:Cho}
 
    Our definition of skew plactic Schur $P$-functions is based on a modification of Sagan and Worley's \cite{Worley, Sagan} rectification algorithm admitting tableaux with low entries on the diagonal. As far as the authors are aware, this variant was first written down explicitly by Cho \cite{Cho}, although it was implicit in earlier works \cite{Worley,Sagan}. Essentially, if no diagonal cell is involved, entries adjacent to a $\bullet$ symbol are compared and the smaller is slid, but when both are equal and low they are now placed in the same column. 
    \begin{equation}
        \begin{ytableau}
        \bullet & \qpr{x} \\
        \qpr{y} & \bullet
    \end{ytableau}
\quad \rightarrow \quad  
    \begin{ytableau}
        \qpr{x} & \bullet \\
        \qpr{y} & \bullet
    \end{ytableau} \quad \text{ for } \qpr{x}<\qpr{y},\text{ and}\hspace{5mm} 
    \begin{ytableau}
        \bullet & \qpr{x} \\
        \qpr{y} & \bullet
    \end{ytableau}
\quad \rightarrow \quad  
    \begin{ytableau}
        \qpr{y} & \qpr{x} \\
        \bullet & \bullet
    \end{ytableau} \quad \text{ for } \qpr{y}<\qpr{x}
    \end{equation}

    \begin{equation}
        \begin{ytableau}
        \bullet & \pr{x} \\
        \pr{x} & \bullet
    \end{ytableau}
\quad \rightarrow \quad  
    \begin{ytableau}
        \pr{x} & \bullet \\
        \pr{x} & \bullet
    \end{ytableau}, \quad \text{ and}\hspace{5mm} 
    \begin{ytableau}
        \bullet & \unpr{x} \\
        \unpr{x} & \bullet
    \end{ytableau}
\quad \rightarrow \quad  
    \begin{ytableau}
        \unpr{x} & \unpr{x} \\
        \bullet & \bullet
    \end{ytableau} 
    \end{equation}
    
    If, on the contrary, the $\bullet$ symbol is in a diagonal cell, we perform a classical jeu de taquin slide unless the entries are equal.
    
    \begin{equation}
        \begin{ytableau}
        \bullet & \qpr{x} \\
        \none & \qpr{y}
    \end{ytableau}
\quad \rightarrow \quad  
    \begin{ytableau}
        \qpr{x} & \bullet \\
        \none & \qpr{y}
    \end{ytableau} \quad \text{ for } \qpr{x}<\qpr{y}
    \end{equation}
    \begin{equation}
        \begin{ytableau}
        \bullet & \pr{x} \\
        \none & \pr{x}
    \end{ytableau}
\quad \rightarrow \quad  
    \begin{ytableau}
        \pr{x} & \unpr{x} \\
        \none & \bullet
    \end{ytableau}, \quad \text{ and}\hspace{5mm} 
    \begin{ytableau}
        \bullet & \pr{x} \\
        \none & \unpr{x}
    \end{ytableau}
\quad \rightarrow \quad  
    \begin{ytableau}
        \unpr{x} & \unpr{x} \\
        \none& \bullet
    \end{ytableau} 
    \end{equation}
Everywhere in the section, ``slide'' stands for a extended Sagan--Worley slide as above. 

We claim that the following resolves Cho's \cref{problem:Cho}.

\begin{definition}\label{def:skewplacticS}
The \newword{skew plactic Schur $P$-function} is
    \[ \sPlacticSchur_{\nu/\mu}\coloneqq  \frac{1}{2^{\diag(\nu/\mu)}} \sum_{T\in \shSSYTprime(\nu/\mu)} [F\circ \rectify(T)]_\sPlactic \in \Z \llbracket \sPlactic \rrbracket,  \]
where $F$ is the map that raises all diagonal positions, $\rectify(\cdot)$ signifies modified Sagan--Worley rectification, and $\diag(\nu/\mu)$ denotes the number of diagonal positions in the skew shape $\nu/\mu.$ (Note that $\diag(\nu / \mu) = \ell(\nu) - \ell(\mu)$.)
\end{definition}

Now we move towards proving that \cref{def:skewplacticS} solves Cho's Open Problem.

    We need some preliminary results clarifying Sagan--Worley's rectification on skew semistandard tableaux which may contain low entries on the diagonal. While it is well known that rectifying shifted standard skew tableaux to a shifted standard target yields $b_{\lambda,\mu}^\nu$ tableaux regardless of the target chosen, the analogous result is not as well explained in the literature for the more general case. Essentially, \cref{thm:generalizedSaganWorley} and \cref{lem:WorleyRectCondition} can be found in \cite{Sagan} and \cite{Worley}, respectively; but the proof in \cite{Worley} is difficult to retrieve and it seems that \cref{lem:WorleyRectCondition} is not used for our same purposes in \cite{Sagan}. These results are also stated in \cite[Proposition~7.3]{Cho} without proof. For the sake of completeness, we include proofs of those results below and use them to show that our definition of shifted plactic skew Schur $P$-functions satisfies Cho's requirements.

\begin{lemma}\label{lem:no_southeast}
    Let $T \in \shSSYTprime(\nu/\lambda)$. Then there is no $\qpr{x}$ SouthEast of another $\qpr{x}$ in $T$. Hence, there is at most one southwestmost $\qpr{x}$ for each $x \in \cN$.
\end{lemma}
\begin{proof}
    The first claim is clear by the increasingness conditions. The second claim follows from the first.
\end{proof}

\begin{lemma}\label{lem:WorleyRectCondition}
    Let $x\in \cN$ be fixed. For an arbitrary skew tableau $S \in \shSSYTprime(\nu / \lambda)$ and $S'$ the result of any sequence of extended Sagan--Worley slides on $S$, the southwestmost $\qpr{x}$ in $S$ is $\unpr{x}$ if and only if the southwestmost $\qpr{x}$ in $S'$ is $\unpr{x}$. 
\end{lemma}
\begin{proof}
By induction, it suffices to assume that $S'$ is obtained from a single slide on $S$.

    Let $x\in \cN$ be such that the southwestmost $\qpr{x}$ in $S$ is $\unpr{x} \in \bbb$. Suppose towards a contradiction that the southwestmost $\qpr{x}$ in $S'$ is $\pr{x}$. By definition, Sagan--Worley slides can change an entry $\pr{x}$ to $\unpr{x}$ but never the opposite; so the slide must displace an $\pr{x}$ westwards without also displacing the southwestmost $\unpr{x}$ westwards. 

  Since a slide can move an entry at most one column left, this $\pr{x}$ must lie in the column of $\bbb$ in $S$. 
    If this $\pr{x} \notin \bbb^\uparrow$ in $S$, then in $S'$ it is NorthWest of the $\unpr{x} \in \bbb$, contradicting \cref{lem:no_southeast}. 
    Hence, $\pr{x} \in \bbb^\uparrow$ in $S$. Therefore the slide displacing this $\pr{x}$ is a diagonal slide, in which the $\pr{x}$ is raised and moves to $\bbb^{\uparrow \leftarrow}$ and $\unpr{x}$ moves up to $\bbb^\uparrow$. But then clearly the southwestmost $\qpr{x}$ is still $\unpr{x}$, a contradiction. 
Thus, if the southwestmost entry $\qpr{x}$ is $\unpr{x}$, after one slide we may have a different southwestmost $\qpr{x}$, but it is still $\unpr{x}$. 

    Now consider $\qpr{x}$ with the southwestmost $\qpr{x}$ being $\pr{x} \in \bbb$. No diagonal slides turning $\pr{x}$ into $\unpr{x}$ are available since there are no instances of $\unpr{x}$ southwest of $\bbb$. Hence, it is enough to consider slides displacing this $\pr{x}$ North and leaving an instance of $\unpr{x}$ South of it. In $S$, this $\unpr{x}$ must be in the row of $\bbb$. Moreover, note that if there is a $\unpr{x}$ in the same row as $\pr{x} \in \bbb$, then there is $\unpr{x} \in \bbb^\rightarrow$ because there is no value $\qpr{e}$ with $\pr{x}<\qpr{e}<\unpr{x}$. Accordingly, in $S$ we have the local configuration
    \ytableausetup{boxsize=normal}
    \[\begin{ytableau}
        \cdots & \cdots & \bullet & \qpr{v} \\
        \none & \cdots & \pr{x}              & \unpr{x}
    \end{ytableau}\]
    and for $\pr{x}$ to move North we must have $\qpr{v}>\pr{x}$. However, $\qpr{v}<\unpr{x}$ by increasingness, so that $\pr{x}<\qpr{v}<\unpr{x}$, an impossibility. Thus the southwestmost $\qpr{x}$ is still $\pr{x}$. 
\end{proof}

\begin{corollary}\label{cor:noprimetarget}
    Let $S \in \shSSYTprime(\nu / \mu)$. Then $\rectify(S) \in \shSSYT(\lambda)$ for some $\lambda$ if and only if $S$ has an $\unpr{x}$ southwest of every $\pr{x}$.
\end{corollary}
\begin{proof}
    This is immediate from \cref{lem:WorleyRectCondition}.
\end{proof}

\begin{theorem}\label{thm:generalizedSaganWorley}
        Let $T \in \shSSYTprime(\lambda)$. Then the number of tableaux $S \in \shSSYTprime(\nu/\mu)$ with $\rectify(S)=T$ is $b_{\lambda,\mu}^\nu.$
\end{theorem}
\begin{proof}
    We establish the result in $x$ special cases and then see that the theorem follows.
    \medskip

    {\sf Case 1 ($T\in \shSSYT(\lambda)$):}
    By \cref{cor:noprimetarget}, the only $S \in  \shSSYTprime(\nu/\mu)$ with $\rectify(S)=T$ are in fact in $\shSSYT(\nu/\mu)$. This case of the theorem then follows from the ordinary version of Sagan--Worley (see Equation $7.6$ in \cite{Worley}.)

\medskip
    {\sf Case 2 ($T \in \shSSYTprime(\lambda)$ and every $\pr{x}$ in $T$ is northeast of some $\unpr{x}$):}

    Let $S \in \shSSYTprime(\nu/\mu)$ with $\rectify(S)=T$. By \cref{lem:WorleyRectCondition}, $S$ has an $\unpr{x}$ southwest of every $\pr{x}$. In particular, for every $x \in \cN$, the southwestmost $\qpr{x}$ in $S$ is $\unpr{x}$. We claim that there exists a bijection between such skew tableaux $S$ of shape $\nu/\mu$ with $\rectify(S)=T$ and standard tableaux $U$ of the same shape with $\rectify(U) = \stand(T)$. 
    
    To wit, compute the standardization of $S$.  This defines an injective map. To show that it is surjective, consider a standard skew tableau $W$. Let $a_i$ be the number of labels in $T$ strictly less than $\pr{i}$. Elements of $W$ in the interval $[a_i+1,a_i+\ct(T)_i]\cap \unpr{\cN}$ form a skew shape $V_i$ inside $W$ and can be partitioned in exactly two ways as a vertical strip whose entries increase from north to south and a horizontal strip whose entries increase from west to east. Choose among those two the partition of $V_i$ whose vertical strip is smallest. For all $i$, change the elements in the vertical strip thus defined of $V_i$ for $\pr{i}$ and in the horizontal strip of $V_i$ for $\unpr{\imath}$. The result is a semistandard skew tableau $W'$ with at least one $\unpr{x}\in \unpr{\cN}$ southwest of each $\pr{x} \in \pr{\cN}$. 

    Observe that performing a slide and standardizing such tableaux $S$ commute as operations on $S$. Accordingly, $\rectify(S)=T$ if and only if $\rectify(\stand(S))=\stand(T).$ But the number of standard skew tableaux of skew shape $\nu/\mu$ rectifying to $\stand(T)$ is $b_{\lambda,\mu}^\nu$. Therefore, by the correspondence between skew standard tableaux of shape $\nu/\mu$, and skew semistandard tableaux of shape $\nu/\mu$ of content $\ct(T)$ with at least a $x$ southwest of $x'$ for every $x'$, we have that the number of the latter rectifying to $T$ is also $b_{\lambda,\mu}^\nu.$ 

\medskip
    {\sf Case 3 ($T\in\shSSYTprime(\lambda)$ arbitrary):}
    
    Suppose now that there are no restrictions on $T$, and let $S$ be of shape $\nu/\mu$ such that $\rectify(S)=T$. By \cref{lem:WorleyRectCondition}, the set $C$ of values $x \in \cN$ such that the southwestmost $\qpr{x}$ in $S$ is $\pr{x}$, is equal to the set of such $x$ where the southwestmost $\qpr{x}$ in $T$ is $\pr{x}$. Accordingly, the map $G$ raising each southwestmost $\qpr{x}$ in $S$ to $\unpr{x}$, furnishes us with a bijection between skew tableaux of shape $\nu/\mu$ rectifying to $T$, and skew tableaux of shape $\nu/\mu$ with all southwestmost  $\qpr{x}$ being $\unpr{x}$ and rectifying to $G(T).$ But the latter were counted; there are $b_{\lambda,\mu}^\nu$ of them. Hence, for general $T$ the theorem follows.
\end{proof}

Note that in our \cref{def:skewplacticS}, Sagan--Worley's algorithm is employed to rectify $T$, but after it has been rectified its shifted plactic class (codifying equivalence under Haiman's insertion) is taken. The following shows that the skew plactic Schur $P$-functions, as defined in \cref{def:skewplacticS}, live in the correct ring and have the desired expansion into Serrano's plactic Schur $P$-functions.

The following shows that our definition solves Cho's Open Problem:
\begin{thm}
    Let $\mu<\nu$ be strict partitions. Then $\sPlacticSchur_{\nu/\mu} \in \mathbb{Q}[\sPlacticSchur_\lambda]_\lambda$ and
    \[\sPlacticSchur_{\nu/\mu}  =  \sum_{\lambda} \frac{2^{\ell(\lambda)}}{2^{\diag(\nu/\mu)}} b_{\lambda,\mu}^\nu  \sPlacticSchur_\lambda.   \]
    In particular, the expansion coefficients of $\sPlacticSchur_{\nu/\mu}$ in the plactic Schur $P$-basis are equal to the expansion coefficients of the ordinary skew Schur $P$-function $P_{\nu/\mu}$ in the ordinary Schur $P$-basis. 
\end{thm}
\begin{proof}
   By \cref{def:skewplacticS}, we need to prove that for any semistandard tableau $S \in \shSSYTprime(\lambda)$, the number of skew tableaux $T \in \shSSYTprime(\nu/\mu)$ with $[F\circ \rectify(T)]_\sPlactic = [F(S)]_\sPlactic$ is equal to $2^{\ell(\lambda)}b_{\lambda,\mu}^\nu$. Noting that $b_{\lambda,\mu}^\nu$ is the number of skew tableaux $T \in \shSSYTprime(\nu/\mu)$ rectifying to any semistandard tableau $U \in \shSSYTprime(\lambda)$ by \cref{thm:generalizedSaganWorley}, it suffices to establish that there are $2^{\ell(\lambda)}$ semistandard tableaux $\hat{S} \in \shSSYTprime(\lambda)$ with $F(\hat{S})=F(S).$ 

    The equality $F(\hat{S})=F(S)$ holds if and only if $\hat{S}$ can be obtained from $F(S)$ by choosing replacing some diagonal entries $\unpr{x}$ with $\pr{x}$. This can be done in $2^{\ell(\lambda)}$ different ways.  
\end{proof}
 
\section*{Acknowledgements}
The authors are grateful for useful conversations with Vic Reiner and Luis Serrano.

Both authors were partially supported by a Discovery Grant (RGPIN-2021-02391) and Launch Supplement (DGECR-2021-00010) from
the Natural Sciences and Engineering Research Council of Canada. SES also acknowledges partial support from a Sinclair Graduate Scholarship from the University of Waterloo.

\bibliographystyle{amsalpha} 
\bibliography{shifted.bib}
\end{document}